\documentclass[12pt]{amsart}
\usepackage[headings]{fullpage}
\usepackage{amssymb,epic,eepic,epsfig,amsbsy,amsmath,amscd,graphicx}
\usepackage[all]{xy}

\numberwithin{equation}{section}
                        \theoremstyle{plain}
\usepackage{mathrsfs}

\newtheorem{theorem}{Theorem}[section]
\newtheorem{thm}{Theorem}
\newtheorem{lemma}[theorem]{Lemma}

\newtheorem{proposition}[theorem]{Proposition}
\newtheorem{conjecture}{Conjecture}

\theoremstyle{definition}

\def\BC{\mathbb C}

\def\BZ{\mathbb Z}

\def\CA{\mathcal A}

\def\CR{\mathcal R}

\def\CT{\mathcal T}

\def\fp{\mathfrak p}
\def\ft{\mathfrak t}

\def\be { \begin{equation} }
\def\ee { \end{equation} }

\begin{document}

\title[AJ conjecture for torus knots]{Proof of a stronger version of the AJ conjecture for torus knots}

\author[Anh T. Tran]{Anh T. Tran}
\address{School of Mathematics, 686 Cherry Street,
 Georgia Tech, Atlanta, GA 30332, USA}
\email{tran@math.gatech.edu}

\thanks{2010 {\em Mathematics Classification:} Primary 57N10. Secondary 57M25.}

\begin{abstract}
For a knot $K$ in $S^3$, the $sl_2$-colored Jones function $J_K(n)$ is a sequence of Laurent polynomials in the variable $t$, which is known to satisfy non-trivial linear recurrence relations. The operator corresponding to the minimal linear recurrence relation is called the recurrence polynomial of $K$. The AJ conjecture \cite{Ga04} states that when reducing $t=-1$, the recurrence polynomial is essentially equal to the $A$-polynomial of $K$. In this paper we consider a stronger version of the AJ conjecture, proposed by Sikora \cite{Si}, and confirm it for all torus knots.
\end{abstract}

\maketitle

\setcounter{section}{-1}

\section{Introduction}

\subsection{The AJ conjecture}

For a knot $K$ in $S^3$, let $J_K(n) \in \mathbb Z[t^{\pm 1}]$ be the colored Jones polynomial of $K$ colored by the $n$-dimensional simple $sl_2$-representation \cite{Jo, RT}, normalized so that
for the unknot $U$,
$$J_U(n) = [n] := \frac{t^{2n}- t^{-2n}}{t^{2} -t^{-2}}.$$
The color $n$ can be assumed to take negative integer values by setting $J_K(-n) = - J_K(n)$. In particular, $J_K(0)=0$. It is known that $J_K(1)=1$, and $J_K(2)$ is the ordinary Jones polynomial.

Define two operators $L,M$ acting on the set of discrete functions $f: \mathbb Z \to \mathbb \CR:=\BC[t^{\pm 1}]$ by
$$(Lf)(n) := f(n+1), \qquad (Mf )(n) := t^{2n} f(n).$$
It is easy to see that $LM= t^2 ML$. Besides, the inverse operators $L^{-1}, M^{-1}$ are well-defined. One can consider $L,M$ as elements of the quantum torus
$$ \mathcal T := \mathbb \CR\langle L^{\pm1}, M^{\pm 1} \rangle/ (LM - t^2ML),$$
which is not commutative, but almost commutative.

Let $$\mathcal A_K = \{ P \in \mathcal T \mid  P J_K=0\},$$ which is a left-ideal of $\mathcal T$, called the {\em recurrence ideal} of $K$. It was proved in \cite{GL} that for every knot $K$, the
recurrence ideal $\mathcal A_K$ is non-zero. An element in $\mathcal A_K$ is called a recurrence relation for the colored Jones polynomials of $K$.

The ring $\mathcal T$ is not a principal left-ideal domain, i.e. not every left-ideal of $\mathcal T$ is generated by one element. By adding all  inverses of polynomials in $t,M$ to
$\mathcal T$ one gets a principal left-ideal domain $\widetilde\CT$, c.f. \cite{Ga04}. The ring $\widetilde{\CT}$ can be formally defined as follows. Let
$\CR(M)$ be the fractional field of the polynomial ring $\CR[M]$.
Let $\widetilde \CT$ be the set of all Laurent polynomials in the
variable $L$ with coefficients in $\CR(M)$:
$$\widetilde
\CT =\{\sum_{j\in \BZ}f_j(M) L^j \,\, | \quad f_j(M)\in \CR(M),
\,\,\, f_j=0  \quad \text{almost everywhere} \},
$$
and define the product in  $\widetilde \CT$ by $f(M) L^{k} \cdot g(M)
L^{l}=f(M)\, g(t^{2k}M) L^{k+l}.$

The left-ideal extension $\widetilde \CA_K :=\widetilde \CT \CA_K$ of $\CA_K$ in
$\widetilde \CT$ is then generated  by a polynomial
$$\alpha_K(t;M,L) = \sum_{j=0}^{d} \alpha_{K,j}(t,M) \, L^j,$$
where $d$ is assumed to be minimal and all the
coefficients $\alpha_{K,j}(t,M)\in \BZ[t^{\pm1},M]$ are assumed to
be co-prime. That $\alpha_K$ can be chosen to have integer
coefficients follows from the fact that $J_K(n) \in
\BZ[t^{\pm1}]$. The polynomial $\alpha_K$ is defined up to a polynomial in $\mathbb Z[t^{\pm 1},M]$. Moreover, one can choose $\alpha_K \in \mathcal A_K$, i.e. it is a recurrence relation for 
the colored Jones polynomials. We will call $\alpha_K$ the {\em recurrence polynomial} of $K$. 

\medskip

Let $\varepsilon$ be the map reducing $t=-1.$ Garoufalidis \cite{Ga04} formulated the following conjecture (see also \cite{FGL, Ge}).

\begin{conjecture}{\bf (AJ conjecture)} For every knot $K$, $\varepsilon(\alpha_K)$ is equal to the $A$-polynomial, up to a polynomial depending on $M$ only.
\label{c1}
\end{conjecture}

The $A$-polynomial of a knot was introduced by Cooper et al. \cite{CCGLS}; it describes the $SL_2(\BC)$-character variety of the knot complement as viewed from the boundary torus. Here in the definition of the $A$-polynomial, we also allow the factor $L-1$ coming from the abelian component of the character variety of the knot group. Hence the $A$-polynomial in this paper is equal to $L-1$ times the $A$-polynomial defined in \cite{CCGLS}.

\medskip

The AJ Conjecture was verified for the trefoil and figure 8 knots by Garoufalidis \cite{Ga04}, and was partially checked for all torus knots by Hikami \cite{Hi}. It was established for some classes of two-bridge knots and pretzel knots, including all twist knots and $(-2,3,6n \pm 1)$-pretzel knots, by Le and the author \cite{Le06, LT}. Here we provide a full proof of the AJ conjecture for all torus knots. Moreover, we show that a stronger version of the conjecture, due to Sikora, holds true for all torus knots.

\subsection{Main results} 

For a finitely generated group $G$, let $\chi(G)$ denote the $SL_2(\BC)$-character variety of $G$, see e.g. \cite{CS, LM}. For a manifold $Y$ we use $\chi(Y)$ also to denote $\chi(\pi_1(Y))$. Suppose  $G=\BZ^2$, the free abelian group with 2 generators.
Every pair  of generators $\mu, \lambda$ will define an isomorphism
between $\chi(G)$ and $(\BC^*)^2/\tau$, where $(\BC^*)^2$ is the
set of non-zero complex pairs $(M,L)$ and $\tau$ is the involution
$\tau(M,L):=(M^{-1},L^{-1})$, as follows: Every representation is
conjugate to an upper diagonal one, with $M$ and $L$ being the
upper left entries of $\mu$ and $\lambda$ respectively. The
isomorphism does not change if one replaces $(\mu, \lambda)$ by
$(\mu^{-1},\lambda^{-1})$.

For an algebraic set $V$ (over $\BC$), let $\BC[V]$ denote the ring
of regular functions on $V$.  For example, $\BC[(\BC^*)^2/\tau]=
\ft^\sigma$, the $\sigma$-invariant subspace of  $\ft:=\BC[M^{\pm
1},L^{\pm 1}]$, where $\sigma(M^kL^l):= M^{-k}L^{-l}.$

\smallskip

Let $K$ be a knot in $S^3$ and $X=S^3 \setminus K$ its complement. 
The boundary of
$X$ is a torus whose fundamental group  is free abelian of rank
two. An orientation of $K$ will define a unique pair of an
oriented meridian $\mu$ and an oriented longitude $\lambda$ such that the linking
number between the longitude and the knot is zero. The pair provides
an identification of $\chi(\partial X)$ and $(\BC^*)^2/\tau$
which actually does not depend on the orientation of $K$.

The inclusion $\partial X \hookrightarrow X$ induces an algebra homomorphism
$$\theta: \BC[\chi(\partial X)]  \equiv \ft^\sigma \longrightarrow
\BC[\chi(X)].$$
We will call the kernel $\fp$ of $\theta$ the {\em $A$-ideal} of $K$; it is an ideal of $\ft^\sigma$. The $A$-ideal was first introduced in \cite{FGL}; it determines the $A$-polynomial of $K$. In fact $\fp=(A_K \cdot \ft)^{\sigma}$, the $\sigma$-invariant part of the ideal $A_K \cdot \ft \subset \ft$ generated by the $A$-polynomial $A_K$.

\smallskip

The involution $\sigma$ acts on the quantum torus $\CT$ also by $\sigma(M^kL^l)= M^{-k}L^{-l}$. Let $\CA_K^\sigma$ be the $\sigma$-invariant part of the recurrence ideal $\CA_K$; it is an ideal of $\CT^{\sigma}$. Sikora \cite{Si} proposed the following conjecture.

\begin{conjecture} 
Suppose $K$ is a knot. Then $\sqrt{\varepsilon(\CA_K^{\sigma})}=\fp$. 
\label{c2}
\end{conjecture}

Here $\sqrt{\varepsilon(\CA_K^{\sigma})}$ denotes the radical of the ideal $\varepsilon(\CA_K^{\sigma})$ in the ring $\ft^{\sigma}=\varepsilon(\CT^{\sigma})$. 

\smallskip

It is easy to see that Conjecture \ref{c2} implies the AJ conjecture. Conjecture \ref{c2} was verified for the unknot and the trefoil knot by Sikora \cite{Si}. In the present paper we confirm it for all torus knots.

\begin{thm}
Conjecture \ref{c2} holds true for all torus knots.
\label{th2}
\end{thm}
\label{intro}

\subsection{Acknowledgements} The author would like to thank T.T.Q. Le for his guidance, S. Garoufalidis for helpful discussions, and the referee for suggestions.

\subsection{Plan of the paper} We provide a full proof of the AJ conjecture for all torus knots in Section \ref{AJ conj} and prove Theorem \ref{th2} in Section 2.

\section{Proof of the AJ conjecture for torus knots}

\label{AJ conj}

We will always assume that knots have framings 0.

\smallskip

Let $T(a,b)$ denote the $(a,b)$-torus knot. We consider the two cases: $a,b>2$ and $a=2$ separately. Lemmas \ref{lemab} and \ref{lem2b} below were first proved in \cite{Hi} using formulas for the colored Jones polynomials and the Alexander polynomial of torus knots given in \cite{Mo}. We present here direct proofs.

\subsection{The case $a,b >2$} 

\begin{lemma}
One has
\begin{eqnarray*}
 J_{T(a,b)}(n+2) = t^{-4ab(n+1)} \, J_{T(a,b)}(n)+t^{-2ab(n+1)} \, \frac{t^{2}\lambda_{(a+b)(n+1)}-t^{-2} \lambda_{(a-b)(n+1)}}{t^{2}-t^{-2}},
\end{eqnarray*}
where $\lambda_k:=t^{2k}+t^{-2k}.$
\label{lemab}
\end{lemma}

\begin{proof}
By \cite{Mo}, we have
\begin{equation}
J_{T(a,b)}(n)=t^{-ab(n^2-1)} \sum_{j=-\frac{n-1}{2}}^{\frac{n-1}{2}} t^{4bj(aj+1)}[2aj+1],
\label{eq0}
\end{equation}
where $[k]:=(t^{2k}-t^{-2k})/(t^2-t^{-2})$.
Hence 
\begin{eqnarray*}
 J_{T(a,b)}(n+2) &=& t^{-ab((n+2)^2-1)} \sum_{j=-\frac{n+1}{2}}^{\frac{n+1}{2}} t^{4bj(aj+1)}[2aj+1]\\
&=& t^{-ab((n+2)^2-1)} \sum_{j=-\frac{n-1}{2}}^{\frac{n-1}{2}} t^{4bj(aj+1)}[2aj+1] + t^{-ab((n+2)^2-1)} \\
&& \times \, \Big( t^{b(n+1)(a(n+1)+2)}[a(n+1)+1]-t^{b(n+1)(a(n+1)-2)}[a(n+1)-1]\Big)\\
      &=&t^{-4ab(n+1)} \, J_{T(a,b)}(n)+t^{-2ab(n+1)} \, \frac{t^2\lambda_{(a+b)(n+1)}-t^{-2}\lambda_{(a-b)(n+1)}}{t^2-t^{-2}}.
\end{eqnarray*}
\end{proof}

\begin{lemma}
The colored Jones function of $T(a,b)$ is annihilated by the operator $F_{a,b}=c_3L^3+c_2L^2+c_1L+c_0$ where
\begin{eqnarray*}
c_3 &:=&  t^2 \Big( t^{2(a+b)}M^{a+b}+t^{-2(a+b)}M^{-(a+b)} \Big)- t^{-2}\Big( t^{2(a-b)}M^{a-b}+t^{-2(a-b)}M^{-(a-b)} \Big),\\
c_2 &:=& - t^{-2ab} \Big( t^2 \big( t^{4(a+b)}M^{a+b}+t^{-4(a+b)}M^{-(a+b)} \big) + t^{-2} \big( t^{4(a-b)}M^{a-b}+t^{-4(a-b)}M^{-(a-b)} \big) \Big),\\
c_1 &:=& - t^{-8ab} M^{-2ab} c_3,\\
c_0 &:=& - t^{-4ab}M^{-2ab} c_2.
\end{eqnarray*}
\label{alpha_ab}
\end{lemma}

\begin{proof}
It is easy to check that $c_3t^{-4ab(n+2)} + c_1 = c_2 t^{-4ab(n+1)} + c_0 = 0$ and
\begin{eqnarray*}
c_3 \big( t^2\lambda_{(a+b)(n+2)}-t^{-2}\lambda_{(a-b)(n+2)} \big)+c_2t^{2ab}\big( t^2\lambda_{(a+b)(n+1)}-t^{-2}\lambda_{(a-b)(n+1)} \big) &=& 0.
\end{eqnarray*}
Hence, by Lemma \ref{lemab},  $F_{a,b} \, J_{T(a,b)}(n)$ is equal to
\begin{eqnarray*}
 && c_3 \, J_{T(a,b)}(n+3) +c_2 \, J_{T(a,b)}(n+2) + c_1 \, J_{T(a,b)}(n+1) + c_0 \, J_{T(a,b)}(n)\\
      &=& c_3 \Big( t^{-4ab(n+2)} \, J_{T(a,b)}(n+1)+t^{-2ab(n+2)} \, \frac{t^2 \lambda_{(a+b)(n+2)}-t^{-2} \lambda_{(a-b)(n+2)}}{t^2-t^{-2}}  \Big)\\
      && + \, c_2 \Big( t^{-4ab(n+1)} \, J_{T(a,b)}(n)+t^{-2ab(n+1)} \, \frac{t^2 \lambda_{(a+b)(n+1)}-t^{-2} \lambda_{(a-b)(n+1)}}{t^2-t^{-2}} \Big) \\
      && + \, c_1 \, J_{T(a,b)}(n+1) + c_0 \, J_{T(a,b)}(n)\\
      &=& (c_3 t^{-4ab(n+2)} + c_1) \, J_{T(a,b)}(n+1) + (c_2 t^{-4ab(n+1)} + c_0) \, J_{T(a,b)}(n) \\
      && + \,  t^{-2ab(n+1)}\Big(  c_3 \, \frac{t^2\lambda_{(a+b)(n+2)}-t^{-2}\lambda_{(a-b)(n+2)}}{t^2-t^{-2}} + c_2t^{2ab}  \, \frac{t^2\lambda_{(a+b)(n+1)}-t^{-2}\lambda_{(a-b)(n+1)}}{t^2-t^{-2}} \Big)\\
      &=& 0.
\end{eqnarray*}
This proves Lemma \ref{alpha_ab}.
\end{proof}

Recall that $\alpha_{T(a,b)}$ is the recurrence polynomial of $T(a,b)$.

\begin{proposition}
For  $a,b>2$, one has $\alpha_{T(a,b)} =F_{a,b}.$
\label{prop}
\end{proposition}

\begin{proof}
By Lemma \ref{alpha_ab} it suffices to show that if an operator $P=P_2L^2+P_1L+P_0$, where $P_j$'s are polynomials in $\BC[t^{\pm 1},M]$, annihilates the colored Jones polynomials of $T(a,b)$ then $P=0.$ 

Indeed, suppose $P \, J_{T(a,b)}(n)=0$. Then, by Lemma \ref{lemab},
\begin{eqnarray*}
0 &=& P_2 \, J_{T(a,b)}(n+2) + P_1 \, J_{T(a,b)}(n+1) + P_0 \, J_{T(a,b)}(n)\\
  &=& P_2 \Big( t^{-4ab(n+1)} \, J_{T(a,b)}(n)+t^{-2ab(n+1)} \, \frac{t^2 \lambda_{(a+b)(n+1)}-t^{-2} \lambda_{(a-b)(n+1)}}{t^2-t^{-2}} \Big) \\
  && + \, P_1 \, J_{T(a,b)}(n+1) + P_0 \, J_{T(a,b)}(n)\\
  &=& (t^{-4ab(n+1)}P_2+P_0) \, J_{T(a,b)}(n) + P_1 \, J_{T(a,b)}(n+1) \\
  && + \, P_2 \, t^{-2ab(n+1)} \, \frac{t^2\lambda_{(a+b)(n+1)}-t^{-2} \lambda_{(a-b)(n+1)}}{t^2-t^{-2}}.
\end{eqnarray*}

Let $P'_2=t^{-4ab(n+1)}P_2+P_0$ and $P'_0=P_2 \, t^{-2ab(n+1)} \, \dfrac{t^2\lambda_{(a+b)(n+1)}-t^{-2}\lambda_{(a-b)(n+1)}}{t^2-t^{-2}}.$ Then 
\begin{equation}
P'_2 \, J_{T(a,b)}(n)+P_1 \, J_{T(a,b)}(n+1)+P'_0=0.
\label{eq1}
\end{equation}
Note that $P'_2$ and $P'_0$ are polynomials in $\BC[t^{\pm 1},M]$. We need the following lemma.

\begin{lemma}
The lowest degree in $t$ of $J_{T(a,b)}(n)$ is $$l_n=-abn^2+ab+\frac{1}{2}(1-(-1)^{n-1})(a-2)(b-2).$$
\label{deg}
\end{lemma}

\begin{proof}
From \eqref{eq0}, it follows easily that $l_n=-abn^2+ab$ if $n$ is odd, and $l_n=(-abn^2+ab)+(ab-2b-2a+4)$ if $n$ is even. 
\end{proof}

Suppose $P'_2, P_1 \not= 0$. Let $r_n$ and $s_n$ be the lowest degrees (in $t$) of $P'_2$ and $P_1$ respectively. Note that, when $n$ is large enough, $r_n$ and $s_n$ are polynomials in $n$ of degrees at most $1$. Equation \eqref{eq1} then implies that $r_n+l_n=s_n+l_{n+1},$ i.e. $$r_n-s_n=l_{n+1}-l_n=-ab(2n+1)-(-1)^n(a-2)(b-2).$$
This cannot happen since the LHS is a polynomial in $n$, when $n$ is large enough, while the RHS is not (since $(a-2)(b-2)>0$). Hence $P'_2=P_1=P'_0=0$, which means $P=0.$ 
\end{proof}

Let us complete the proof of Proposition \ref{prop}. It is easy to see that $\varepsilon(\alpha_{T(a,b)})=M^{-2ab}(M^a-M^{-a})(M^b-M^{-b})A_{T(a,b)}$ where $A_{T(a,b)}=(L-1)(L^2M^{2ab}-1)$ is the $A$-polynomial of $T(a,b)$ when $a,b>2$. This means the AJ conjecture holds true for $T(a,b)$ when $a,b>2.$

\subsection{The case $a=2$} 

\begin{lemma} One hase
\begin{eqnarray*} 
J_{T(2,b)}(n+1) &=& -t^{-(4n+2)b} \, J_{T(2,b)}(n)+t^{-2nb}[2n+1].
\end{eqnarray*}
\label{lem2b}
\end{lemma}

\begin{proof}
By \eqref{eq0}, we have
$$J_{T(2,b)}(n)=t^{-2b(n^2-1)} \sum_{j=-\frac{n-1}{2}}^{\frac{n-1}{2}} t^{4bj(2j+1)}[4j+1].$$
Hence
\begin{eqnarray*}
J_{T(2,b)}(n+1)=t^{-2b((n+1)^2-1)} \sum_{k=-\frac{n}{2}}^{\frac{n}{2}} t^{4bk(2k+1)}[4k+1]
\end{eqnarray*}
Set $k=-(j+\frac{1}{2}).$ Then
\begin{eqnarray*}
J_{T(2,b)}(n+1) &=& t^{-2b((n+1)^2-1)} \sum_{j=\frac{n-1}{2}}^{-\frac{n+1}{2}} t^{4bj(2j+1)}[-(4j+1)]\\
      &=& t^{-2b((n+1)^2-1)} \Big( -\sum_{j=-\frac{n-1}{2}}^{\frac{n-1}{2}} t^{4bj(2j+1)}[4j+1]+t^{2bn(n+1)}[2n+1] \Big)\\
      &=&-t^{-(4n+2)b} \, J_{T(2,b)}(n)+t^{-2nb}[2n+1].
\end{eqnarray*}
This proves Lemma \ref{lem2b}.
\end{proof}

\begin{lemma}
The colored Jones function of $T(2,b)$ is annihilated by the operator $G_{2,b}=d_2L^2+d_1L+d_0$ where
\begin{eqnarray*}
d_2 &:=& t^2 M^2 - t^{-2}M^{-2},\\
d_1 &:=& t^{-2b} \Big( t^{-4b}M^{-2b} (t^2 M^2 - t^{-2} M^{-2}) - (t^6 M^2 - t^{-6} M^{-2}) \Big),\\
d_0 &:=& -t^{-4b} M^{-2b}(t^6 M^2 - t^{-6} M^{-2}).
\end{eqnarray*}
\label{alpha_2b}
\end{lemma}

\begin{proof}
From Lemma \ref{lem2b} we have
\begin{eqnarray*}
J_{T(2,b)}(n+1) &=& -t^{-(4n+2)b} \, J_{T(2,b)}(n)+t^{-2nb}[2n+1],\\
J_{T(2,b)}(n+2) &=& t^{-8(n+1)b} J\, _{T(2,b)}(n)-t^{-6(n+1)b}[2n+1]+t^{-2(n+1)b}[2n+3].
\end{eqnarray*}
It is easy to check that 
\begin{eqnarray*}
t^{-8(n+1)b}d_2-t^{-(4n+2)b}d_1+d_0 &=& 0,\\
d_2 \Big( -t^{-6(n+1)b}[2n+1]+t^{-2(n+1)b}[2n+3] \Big) + d_1 t^{-2nb}[2n+1] &=& 0.
\end{eqnarray*}
Hence
\begin{eqnarray*}
G_{2,b} \, J_{T(2,b)}(n) &=& d_2 \, J_{T(2,b)}(n+2) + d_1 \, J_{T(2,b)}(n+1) + d_0 \, J_{T(2,b)}(n)\\
                  &=& d_2 \Big( t^{-8(n+1)b} \, J_{T(2,b)}(n)-t^{-6(n+1)b}[2n+1]+t^{-2(n+1)b}[2n+3] \Big)\\
                  && + \, d_1 \Big( -t^{-(4n+2)b} \, J_{T(2,b)}(n)+t^{-2nb}[2n+1] \Big) + d_0 \, J_{T(2,b)}(n)\\
                  &=& \Big( t^{-8(n+1)b}d_2-t^{-(4n+2)b}d_1+d_0 \Big) \, J_{T(2,b)}(n) \\
                  && + \, d_2 \Big( -t^{-6(n+1)b}[2n+1]+t^{-2(n+1)b}[2n+3] \Big) + d_1 t^{-2nb}[2n+1]\\
                  &=& 0.
\end{eqnarray*}
This proves Lemma \ref{alpha_2b}.
\end{proof}

\begin{proposition}
One has $\alpha_{T(2,b)} =G_{2,b}.$
\end{proposition}

\begin{proof}
By Lemma \ref{alpha_2b}, it suffices to show that if an operator $P=P_1L+P_0$, where $P_j$'s are polynomials in $\BC[t^{\pm 1},M]$, annihilates the colored Jones polynomials of $T(2,b)$ then $P=0.$ 

Indeed, suppose $P \, J_{T(2,b)}(n)=0$. Then
\begin{eqnarray*}
0 &=& P_1 \, J_{T(2,b)}(n+1) + P_0 \, J_{T(2,b)}(n)\\
  &=& P_1 \Big( -t^{-(4n+2)b} \, J_{T(2,b)}(n)+t^{-2nb}[2n+1] \Big) + P_0 \, J_{T(2,b)}(n)\\
  &=& \big(-t^{-(4n+2)b}P_1+P_0\big) \, J_{T(2,b)}(n) + t^{-2nb}[2n+1] P_1.
\end{eqnarray*}
Let $P'_1=-t^{-(4n+2)b}P_1+P_0$ and $P'_0=t^{-2nb}[2n+1] P_1$. Then $P'_1, P'_0$ are polynomials in $\BC[t^{\pm 1},M]$ and $P_1'J(n)+P'_0=0.$ This implies that $P_1'=P_0'=0$ since the lowest degree in $t$ of $J_{T(2,b)}(n)$ is $-2bn^2+2b$, which is quadratic in $n$, by Lemma \ref{deg}. Hence $P=0.$ 
\end{proof}

It is easy to see that $\varepsilon(\alpha_{T(2,b)})=M^{-2b}(M^2-M^{-2})A_{T(2,b)}$ where $A_{T(2,b)}=(L-1)(LM^{2b}+1)$ is the $A$-polynomial of $T(2,b)$. This means the AJ conjecture holds true for $T(2,b)$.

\section{Proof of Theorem \ref{th2}}

\label{ver}

As in the previous section, we consider the two cases: $a,b>2$ and $a=2$ separately.

\subsection{The case $a,b >2$}  We claim that

\begin{proposition}
 The colored Jones function of $T(a,b)$ is annihilated by the operator $PQ$ where
\begin{eqnarray*}
P &=& t^{-10ab}(L^3M^{2ab}+L^{-3}M^{-2ab})-(t^{2(a-b)}+t^{2(b-a)})t^{-4ab}(L^2M^{2ab}+L^{-2}M^{-2ab})\\
&&+\,t^{2ab}(LM^{2ab}+L^{-1}M^{-2ab})-(t^{2ab}+t^{-2ab})(L+L^{-1})\\
&& + \, (t^{2(a-b)}+t^{2(b-a)})(t^{4ab}+t^{-4ab}),\\
Q &=& t^{-6ab}(L^3M^{2ab}+L^{-3}M^{-2ab})-(t^{2(a+b)}+t^{-2(a+b)})q^{-ab}(L^2M^{2ab}+L^{-2}M^{-2ab})\\
&&+\,t^{-2ab}(LM^{2ab}+L^{-1}M^{-2ab})-(t^{2ab}+t^{-2ab})(L+L^{-1})+2(t^{2(a+b)}+t^{-2(a+b)}).
\end{eqnarray*}
\label{PQ}
\end{proposition}

\begin{proof} We first prove the following two lemmas.

\begin{lemma}
One has $$Q \, J_{T(a,b)}(n)=t^{4ab-2}(\lambda_{a+b}-\lambda_{a-b}) \, \frac{t^{2abn}\lambda_{(a-b)(n+1)}-t^{-2abn}\lambda_{(a-b)(n-1)}}{t^2-t^{-2}}.$$
\label{Q}
\end{lemma}

\begin{proof}
Let 
$$g(n) = t^{-2abn} \, \frac{t^2 \lambda_{(a+b)n}-t^{-2} \lambda_{(a-b)n}}{t^2-t^{-2}}.$$
Then, by Lemma \ref{lemab}, $J_{T(a,b)}(n+2) = t^{-4ab(n+1)} \, J_{T(a,b)}(n)+g(n+1)$. Hence $Q\,J_{T(a,b)}(n)$ is equal to
\begin{eqnarray*}
  && t^{-6ab} \Big( t^{4ab(n+3)} \, J_{T(a,b)}(n+3)+t^{-4ab(n-3)} \, J_{T(a,b)}(n-3) \Big)\\
       && - \, (t^{2(a+b)}+t^{-2(a+b)})t^{-4ab} \Big( t^{4ab(n+2)} \, J_{T(a,b)}(n+2)+t^{-4ab(n-2)} \, J_{T(a,b)}(n-2) \Big)\\
       && + \, t^{-2ab} \Big( t^{4ab(n+1)} \, J_{T(a,b)}(n+1)+t^{-4ab(n-1)} \, J_{T(a,b)}(n-1) \Big)\\
       && - \, (t^{2ab}+t^{-2ab}) \Big( J_{T(a,b)}(n+1)+J_{T(a,b)}(n-1) \Big)+2(t^{2(a+b)}+t^{-2(a+b)}) \, J_{T(a,b)}(n)\\
       &=& t^{-6ab} \Big( t^{4ab} \big( J_{T(a,b)}(n+1)+J_{T(a,b)}(n-1) \big)+t^{2ab(n+5)}g(n+2)-t^{-2ab(n-5)}g(n-2) \Big)\\
       && - \, (t^{2(a+b)}+t^{-2(a+b)})t^{-4ab} \Big( 2 t^{4ab} \, J_{T(a,b)}(n)+t^{2ab(n+4)}g(n+1)-t^{-2ab(n-4)}g(n-1) \Big)\\
       && + \,  t^{-2ab} \Big( t^{4ab} \left(J_{T(a,b)}(n-1)+J_{T(a,b)}(n+1)\right)+\left(t^{2ab(n+3)}-t^{-2ab(n-3)}\right)g(n) \Big)\\
       && - \, (t^{2ab}+t^{-2ab}) \Big( J_{T(a,b)}(n+1)+J_{T(a,b)}(n-1) \Big)+2(t^{2(a+b)}+t^{-2(a+b)}) \, J_{T(a,b)}(n)\\
       &=& t^{-6ab} \Big( t^{2ab(n+5)}g(n+2)-t^{-2ab(n-5)}g(n-2) \Big) \\
       && -\, \left(t^{2(a+b)}+t^{-2(a+b)}\right) t^{-4ab}\Big( t^{2ab(n+4)}g(n+1)-t^{-2ab(n-4)}g(n-1) \Big)\\
       && + \, t^{-2ab}\left(t^{2ab(n+3)}-t^{-2ab(n-3)}\right)g(n).
\end{eqnarray*}
Using the definition of $g(n)$, $Q \, J_{T(a,b)}(n)$ is equal to
\begin{eqnarray*}
  && t^{4ab} \Big( t^{2abn} \, \frac{t^{2}\lambda_{(a+b)(n+2)}-t^{-2}\lambda_{(a-b)(n+2)}}{t^2-t^{-2}}-t^{-2abn} \, \frac{t^{2} \lambda_{(a+b)(n-2)}-t^{-2} \lambda_{(a-b)(n-2)}}{t^2-t^{-2}} \Big)\\
 && -\, (t^{2(a+b)}+t^{-2(a+b)}) t^{4ab} \times \\
&& \Big( t^{2abn} \, \frac{t^2\lambda_{(a+b)(n+1)}-t^{-2}\lambda_{(a-b)(n+1)}}{t^2-t^{-2}} -t^{-2abn} \, \frac{t^2\lambda_{(a+b)(n-1)}-t^{-2} \lambda_{(a-b)(n-1)}}{t^2-t^{-2}} \Big)\\
 && + \, t^{4ab} (t^{2abn}-t^{-2abn})  \, \frac{t^2\lambda_{(a+b)n}-t^{-2}\lambda_{(a-b)n}}{t^2-t^{-2}}.
\end{eqnarray*}
Now applying the equality $\lambda_{k+l}+\lambda_{k-l}=\lambda_{k}\lambda_{l}$, we then obtain
$$Q \, J_{T(a,b)}(n)=t^{4ab-2}(\lambda_{a+b}-\lambda_{a-b}) \, \frac{t^{2abn}\lambda_{(a-b)(n+1)}-t^{-2abn}\lambda_{(a-b)(n-1)}}{t^2-t^{-2}}.$$
This proves Lemma \ref{Q}.
\end{proof}

Let $h(n)=t^{2abn}\lambda_{(a-b)(n+1)}-t^{-2abn}\lambda_{(a-b)(n-1)}$. 

\begin{lemma}
 The function $h(n)$ is annihilated by the operator $P$, i.e. $Ph(n)=0.$
 \label{P}
\end{lemma}
\begin{proof}
Let $c=a-b$. Then $Ph(n)$ is equal to
\begin{eqnarray*} 
&&t^{-10ab} \Big( t^{4ab(n+3)}h(n+3)+t^{-4ab(n-3)}h(n-3) \Big)\\
&& - \, (t^{2(a-b)}+t^{2(b-a)})t^{-4ab} \Big(t^{4ab(n+2)}h(n+2)+t^{-4ab(n-2)}h(n-2) \Big)\\
&&+\,t^{2ab} \Big( t^{4ab(n+1)}h(n+1)+t^{-4ab(n-1)}h(n-1) \Big)\\
&& - \, (t^{2ab}+t^{-2ab}) \Big( h(n+1)+h(n-1) \Big)+(t^{2(a-b)}+t^{2(b-a)})(t^{4ab}+t^{-4ab})h(n) \\
&=& \Big( t^{2ab(3n+4)}\lambda_{c(n+4)}-t^{2ab(n-2)}\lambda_{c(n+2)}+t^{-2ab(n+2)} \lambda_{c(n-2)}-t^{-2ab(3n-4)}\lambda_{c(n-4)}\Big)\\
&& - \, \lambda_{c} \Big(t^{2ab(3n+4)}\lambda_{c(n+3)}-t^{2abn}\lambda_{c(n+1)}+t^{-2abn}\lambda_{c(n-1)}-t^{-2ab(3n-4)}\lambda_{c(n-3)}\Big)\\
&& + \, \Big(t^{2ab(3n+4)}\lambda_{c(n+2)}-t^{2ab(n+2)}\lambda_{cn}+t^{-2ab(n-2)}\lambda_{cn}-t^{-2ab(3n-4)}\lambda_{c(n-2)}\Big)\\
                                && - \, (t^{2ab}+t^{-2ab})\Big(t^{2ab(n+1)}\lambda_{c(n+2)}-t^{-2ab(n+1)}\lambda_{cn} + t^{2ab(n-1)}\lambda_{cn}-t^{-2ab(n-1)}\lambda_{c(n-2)}\Big)\\
&& + \, \lambda_{c}(t^{4ab}+t^{-4ab}) \Big(t^{2abn}\lambda_{c(n+1)}-t^{-2abn}\lambda_{c(n-1)}\Big).
\end{eqnarray*}

Note that $\lambda_{k+l}+\lambda_{k-l}=\lambda_{k}\lambda_{l}$. Hence $Ph(n)$ is equal to
\begin{eqnarray*}
 &&  \Big(-t^{2ab(n-2)}\lambda_{c(n+2)}+t^{-2ab(n+2)}\lambda_{c(n-2)}\Big)-\lambda_{c}\Big(-t^{2abn}\lambda_{c(n+1)}+t^{-2abn}\lambda_{c(n-1)}\Big)\\
&& + \, \Big(-t^{2ab(n+2)}\lambda_{cn}+ t^{-2ab(n-2)}\lambda_{cn}\Big)\\
                                && - \, (t^{2ab}+t^{-2ab})\Big(t^{2ab(n+1)}\lambda_{c(n+2)}-t^{-2ab(n+1)}\lambda_{cn} +t^{2ab(n-1)}\lambda_{cn}-t^{-2ab(n-1)}\lambda_{c(n-2)}\Big)\\
&& + \, \lambda_{c}(t^{4ab}+t^{-4ab}) \Big(t^{2abn}\lambda_{c(n+1)}-t^{-2abn}\lambda_{c(n-1)}\Big)\\
                                &=& -(t^{4ab}+t^{-4ab}+1)t^{2abn}\lambda_{c(n+2)}+(t^{4ab}+t^{-4ab}+1)t^{-2abn}\lambda_{c(n-2)}\\
                                && - \, (t^{4ab}+t^{-4ab}+1) (t^{2abn}-t^{-2abn})\lambda_{cn}\\
                                && + \,\lambda_{c}(t^{4ab}+t^{-4ab}+1) \Big(t^{2abn}\lambda_{c(n+1)}-t^{-2abn}\lambda_{c(n-1)}\Big)\\
                                &=& -(t^{4ab}+t^{-4ab}+1)t^{2abn}\Big(\lambda_{c(n+2)}+\lambda_{cn}-\lambda_{c}\lambda_{c(n+1)}\Big)\\
                                && + \, (t^{4ab}+t^{-4ab}+1) t^{-2abn}\Big(\lambda_{c(n-2)}+\lambda_{cn}-\lambda_{c}\lambda_{c(n-1)}\Big)\\
                                &=& 0.
\end{eqnarray*}
This proves Lemma \ref{P}.
\end{proof}

Proposition \ref{PQ} follows directly from Lemmas \ref{Q} and \ref{P}.
\end{proof}

\subsection{The case $a=2$} We claim that

\begin{proposition}
The colored Jones function of $T(2,b)$ is annihilated by the operator
\begin{eqnarray*}
R &=& t^{-4b}(L^2M^{2b}+L^{-2}M^{-2b})+(t^{2b}+t^{-2b})(L+L^{-1})\\&&-(t^4+t^{-4})t^{-2b}(LM^{2b}+L^{-1}M^{-2b})+(M^{2b}+M^{-2b})-2(t^4+t^{-4}).
\end{eqnarray*}
\label{R}
\end{proposition}

\begin{proof}
From Lemma \ref{lem2b} we have
\begin{eqnarray*}
J_{T(2,b)}(n+1) &=& -t^{-(4n+2)b}J_{T(2,b)}(n)+t^{-2nb}[2n+1],\\
J_{T(2,b)}(n+2) &=& t^{-8(n+1)b} J_{T(2,b)}(n)-t^{-6(n+1)b}[2n+1]+t^{-2(n+1)b}[2n+3],\\
J_{T(2,b)}(n-1) &=& -t^{(4n-2)b}J_{T(2,b)}(n)+t^{2nb}[2n-1],\\
J_{T(2,b)}(n-2) &=& t^{8(n-1)b} J_{T(2,b)}(n)-t^{6(n-1)b}[2n-1]+t^{2(n-1)b}[2n-3].
\end{eqnarray*}
Hence $R \, J_{T(2,b)}(n)$ is equal to
\begin{eqnarray*}
 && t^{-4b} \Big( t^{4(n+2)b} \, J_{T(2,b)}(n+2)+t^{-4(n-2)b} \, J_{T(2,b)}(n-2) \Big)\\
 && + \, (t^{2b}+t^{-2b})\Big( J_{T(2,b)}(n+1)+J_{T(2,b)}(n-1) \Big)\\&&-(t^4+t^{-4})t^{-2b} \Big( t^{4(n+1)b} \, J_{T(2,b)}(n+1)+t^{-4(n-1)b} \, J_{T(2,b)}(n-1) \Big)\\
&&+\, \Big( (t^{4nb}+t^{-4nb})-2(t+t^{-4}) \Big)J_{T(2,b)}(n)\\
&=& t^{-4b} \Big( -t^{-2(n-1)b}[2n+1]+t^{2(n+3)b}[2n+3]-t^{2(n+1)b}[2n-1]+t^{-2(n-3)b}[2n-3] \Big)\\
&&+(t^{2b}+t^{-2b}) \Big( t^{-2nb}[2n+1]+t^{2nb}[2n-1] \Big)\\
&&-(t^4+t^{-4})t^{-2b} \Big( t^{(2n+4)b}[2n+1]+t^{(-2n+4)b}[2n-1] \Big)\\
&&-(t+t^{-4})t^{2b}\Big(t^{2nb}[2n+1]+t^{-2nb}[2n-1]\Big)\\
&=&t^{2b}t^{2nb} \Big( [2n+3]+[2n-1]-(t^4+t^{-4})[2n+1] \Big)\\
&&+\, t^{2b}t^{-2nb} \Big( [2n-3]+[2n+1]-(t^4+t^{-4})[2n-1] \Big)\\
&=& 0,
\end{eqnarray*}
since $[k+l]+[k-l]=(t^{2l}+t^{-2l})[k].$ 
\end{proof}

\subsection{Proof of Theorem \ref{th2}} We first note that the $A$-ideal $\fp$, the kernel of $\theta: \ft^\sigma \longrightarrow
\BC[\chi(X)]$, is radical i.e. $\sqrt{\fp}=\fp$, since the character ring $\BC[\chi(X)]$ is reduced, i.e. has nil-radical 0, by definition.

\begin{lemma}
Suppose $\delta(t,M,L) \in \CA_K$. Then there are polynomials $g(t,M) \in \BC[t^{\pm 1},M]$ and $\gamma(t,M,L) \in \CT$ such that
\begin{equation}
\delta(t,M,L)=\frac{1}{g(t,M)}\,\gamma(t,M,L)\,\alpha_K(t,M,L).
\label{delta}
\end{equation}
Moreover, $g(t,M)$ and $\gamma(t,M,L)$ can be chosen so that $\varepsilon (g) \not= 0$. 
\label{00}
\end{lemma}

\begin{proof}
By definition $\alpha_K$ is a generator of $\widetilde{\CA}_K$, the extension of $\CA_K$ in the principal left-ideal domain $\widetilde{\CT}$. Since $\delta \in \CA_K$, it is divisible by $\alpha_K$ in $\widetilde{\CT}$. Hence \eqref{delta} follows.

We can assume that $t+1$ does not divide both $g(t,M)$ and $\gamma(t,M,L)$ simultaneously. If $\varepsilon(g)=0$ then $g$ is divisible by $t+1$, and hence $\gamma$ is not. But then from the equality $g\delta=\gamma\alpha_K$, it follows that $\alpha_K$ is divisible by $t+1$, which is impossible, since all the coefficients of powers of $L$ in $\alpha_K$ are supposed to be co-prime.
\end{proof}

\underline{\em Showing $\, \sqrt{\varepsilon(\CA_K^{\sigma})} \subset \fp$}. For torus knots, by Section 1, we have $\varepsilon(\alpha_K)=f(M)A_K$, where $f(M) \in \BC[M^{\pm 1}]$. For every $\delta \in \CA_K$, by Lemma \ref{00}, there exist $g(t,M) \in \BC[t^{\pm 1},M]$ and $\gamma \in \CT$ such that $\delta=\frac{1}{g(t,M)}\,\gamma\,\alpha_K$ and $\varepsilon (g) \not= 0$. It implies that
\begin{equation}
\varepsilon(\gamma) = \frac{1}{\varepsilon (g(M))}\, \varepsilon(\gamma)\, \varepsilon(\alpha_K)=\frac{1}{\varepsilon (g(M))}\, \varepsilon(\gamma) \, f(M)A_K.
\label{Mfactor}
\end{equation}

The $A$-polynomial of a torus knot does not contain any non-trivial factor depending on $M$ only. Since $\varepsilon(\gamma) \in \ft=\BC[L^{\pm 1},M^{\pm 1}]$, equation \eqref{Mfactor} implies that $h:=\frac{1}{\varepsilon (g(M))}\, \varepsilon(\gamma) \, f(M)$ is an element of $\ft$. Hence $\varepsilon(\gamma) \in A_K \cdot \ft$, the ideal of $\ft$ generated by $A_K$. It follows that $\varepsilon(\CA_K) \subset A_K \cdot \ft$ and thus  $\varepsilon(\CA_K^{\sigma}) \subset (A_K \cdot \ft)^{\sigma}=\fp$. Hence $\sqrt{\varepsilon(\CA_K^{\sigma})} \subset \sqrt{\fp}=\fp$. 

\underline{\em Showing $\, \fp \subset \sqrt{\varepsilon(\CA_K^{\sigma})}$}. For $a,b>2$, by Proposition \ref{PQ} the colored Jones function of $T(a,b)$ is annihilated by the operator $PQ.$ Note that \begin{eqnarray*}
\varepsilon(PQ) &=& (L+L^{-1}-2)^2(L^2M^{2ab}+L^{-2}M^{-2ab}-2)^2\\
&=& L^{-2} \big( L^{-1}M^{-ab}(L-1)(L^2M^{2ab}-1) \big)^4.
\end{eqnarray*}
If $u \in \fp$ then $u=vA'_{T(a,b)}$, where $A'_{T(a,b)}:=L^{-1}M^{-ab}(L-1)(L^2M^{2ab}-1)=L^{-1}M^{-ab}A_{T(a,b)}$ and $v \in \BC[M^{\pm 1}, L^{\pm 1}]$. It is easy to see that $\sigma(v)=Lv$, since $\sigma(u)=u$ and $\sigma(A'_{T(a,b)})=L^{-1}A'_{T(a,b)}.$ This implies that $\sigma(v^2L)=\sigma(v)^2L^{-1}=v^2L$. We then have $$u^4=v^4A_{T(a,b)}^{'4}= \varepsilon(v^4L^2PQ) \in \varepsilon(\CA_K^{\sigma}),$$ hence $u \in  \sqrt{\varepsilon(\CA_K^{\sigma})}.$

For $a=2$, by Proposition \ref{R} the colored Jones function of $T(2,b)$ is annihilated by the operator $R.$ Note that $\sigma(R)=R$ and 
\begin{eqnarray*}
\varepsilon(R) &=& (L+L^{-1}-2)(LM^{2b}+L^{-1}M^{-2b}+2)\\
&=& \big( L^{-1}M^{-b}(L-1)(LM^{2b}+1) \big)^2.
\end{eqnarray*}
If $u \in \fp$ then $u=vA'_{T(2,b)}$, where $A'_{T(2,b)}:=L^{-1}M^{-b}(L-1)(LM^{2b}+1)=L^{-1}M^{-b}A_{T(2,b)}$ and $v \in \BC[M^{\pm 1}, L^{\pm 1}]$. It is easy to see that $\sigma(v)=-v$ and hence $\sigma(v^2)=\sigma(v)\sigma(v)=v^2$. We then have $$u^2=v^2A_{T(2,b)}^{'2}= \varepsilon(v^2R) \in \varepsilon(\CA_K^{\sigma}),$$ hence $u \in  \sqrt{\varepsilon(\CA_K^{\sigma})}.$

In both cases $\fp \subset \sqrt{\varepsilon(\CA_K^{\sigma})}.$ Hence $\sqrt{\varepsilon(\CA_K^{\sigma})}=\fp$ for all torus knots.

\end{document}